\documentclass[11pt]{elsarticle}
\usepackage{graphicx}
\usepackage{amsmath}
\usepackage{amssymb}
\usepackage{amscd}
\usepackage{amsthm}
\usepackage{color}
\usepackage{enumerate}
\usepackage{caption}
\usepackage{subcaption}
\usepackage[top=1in, bottom=1in, left=1in, right=1in]{geometry}

\newtheorem{theorem}{Theorem}
\newtheorem{lemma}[theorem]{Lemma}
\newtheorem{remark}{Remark}

\title{Mixed mode oscillations in a conceptual climate model}

\author[cornell]{Andrew Roberts \corref{cor}}
\ead{andrew.roberts@cornell.edu}
\address[cornell]{Cornell University, Department of Mathematics, \\ 310 Malott Hall, \\ Ithaca, NY 14853, USA}
\cortext[cor]{Corresponding author.  Phone: (607) 255-4013  Email: andrew.roberts@cornell.edu}

\author[uhwo]{Esther Widiasih}
\ead{widiasih@hawaii.edu}
\address[uhwo]{University of Hawai'i - West O'ahu, Department of Mathematics, \\ 91-1001 Farrington Hwy, Library Room \\ 203 Kapolei, HI 96707, USA}

\author[usyd]{Martin Wechselberger}
\ead{wm@maths.usyd.edu.au}
\address[usyd]{University of Sydney, School of Mathematics and Statistics \\ School of Mathematics and Statistics F07 \\ University of Sydney NSW 2006, Australia}

\author[unc]{Chris K. R. T. Jones}
\address[unc]{University of North Carolina, Department of Mathematics, \\ Phillips Hall, CB \#3250, \\ Chapel Hill, NC 27599-3250, USA}
\ead{ckrtj@email.unc.edu}

\begin{document}
\begin{abstract}
Much work has been done on relaxation oscillations and other simple oscillators in conceptual climate models.  However, the oscillatory patterns in climate data are often more complicated than what can be described by such mechanisms.  This paper examines complex oscillatory behavior in climate data through the lens of mixed-mode oscillations.  As a case study, a conceptual climate model with governing equations for global mean temperature, atmospheric carbon, and oceanic carbon is analyzed.  The nondimensionalized model is a fast/slow system with one fast variable (corresponding to  ice volume) and two slow variables (corresponding to the two carbon stores).  Geometric singular perturbation theory is used to demonstrate the existence of a folded node singularity.  A parameter regime is found in which (singular) trajectories that pass through the folded node are returned to the singular funnel in the limiting case where $\epsilon = 0$.  In this parameter regime, the model has a stable periodic orbit of type $1^s$ for some $s>0$.  To our knowledge, it is the first conceptual climate model demonstrated to have the capability to produce an MMO pattern.
\end{abstract}


\maketitle

\section{Introduction}
There has been a significant amount of research aimed at explaining oscillations in various historical periods of the climate system.  Saltzman and Maasch have a series of papers on the Mid-Pleistocene transition, a change from oscillations with a dominant period of 40 kyr to oscillations with a dominant period of 100 kyr \cite{ms88, ms90, ms91}.  According to Saltzman and Maasch , the 40 kyr oscillations in the data result from a linear response to (quasi-)periodic changes in astronomical forcing.  They propose that the transition to the 100 kyr cycles occurs due to a Hopf bifurcation producing an attracting periodic orbit.  Paillard and Parrenin also seek to explain the Mid-Pleistocene transition and the glacial-interglacial cycles of the late Pleistocene, with a discontinuous and piecewise linear model \cite{pp04}. Their work, and the work of Hogg \cite{hogg}, use changes in astronomical forcing due to variation in the Earth's orbit to generate oscillations.  Crucifix \cite{crucifix} and Ditlevsen \cite{ditlevsen09} review oscillations in conceptual climate models.  In particular, Crucifix \cite{crucifix} discusses relaxation oscillators in ice-age models.  However, to our knowledge, the discussion is limited to single-amplitude or single mode oscillations.

\begin{figure}[t]
\begin{centering}
\includegraphics[width=\textwidth]{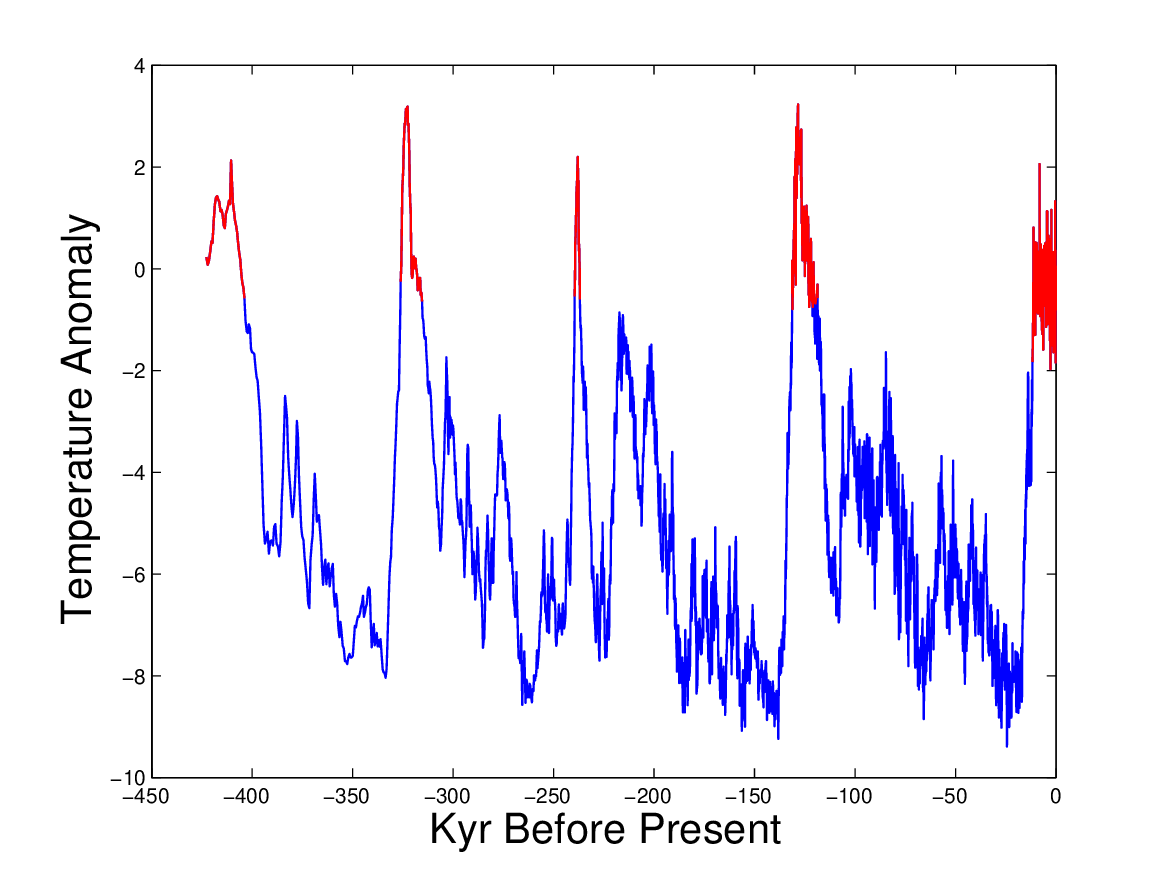}
\caption{Temperature anomaly obtained from the Vostok ice core deuterium record \cite{petit2000}.}
\label{mmoFig:tempRecord}
\end{centering}
\end{figure}

Looking at Figure \ref{mmoFig:tempRecord}, each 100 kyr cycle contains a sharp increase leading into the interglacial period (denoted by the red spikes).  This relaxation behavior clearly indicates the existence of multiple time-scales in the underlying problem.  There are also smaller, structured oscillations in the glacial state that are repeated in each 100 kyr cycle.  The presence of the large relaxation oscillation and the small amplitude oscillations indicates that these may be mixed-mode oscillations (MMOs)---a pattern of $L_1$ large amplitude oscillations (LAOs) followed by $s_1$ small amplitude oscillations (SAOs), then $L_2$ large spikes, $s_2$ small cycles, and so on.  The sequence ${L_1}^{s_1}{L_2}^{s_2}{L_3}^{s_3} \dots$ is known as the MMO signature \cite{mmosurvey}.  We propose that the 100 kyr glacial-interglacial cycles as well as the largest of the SAOs---i.e., the largest cycles that do not enter the interglacial state---can be interpreted as MMOs.  By suggesting that the SAOs result from intrinsic dynamics, we are proposing an alternative to the standard interpretation that attributes them to changes in astronomical forcing.  We note, however, that neither interpretation necessarily rules out the other---it may be possible to combine both intrinsic and forced oscillations.

This paper tests the scientific hypothesis that oscillatory behavior in climate data can be interpreted as MMOs, and we take the data in Figure \ref{mmoFig:tempRecord} as a case study.  Desroches et al. survey the mechanisms that can produce MMOs in systems with multiple time-scales \cite{mmosurvey}.  From the data set shown in Figure \ref{mmoFig:tempRecord}, we know that the underlying model has a multiple time-scale structure.  If we want to find MMOs, the model must have at least three state variables.  Assuming we can find a global time-scale splitting, there are three distinct ways to have a 3D model with multiple time scales: (a) 1 fast, 2 slow; (b) 2 fast, 1 slow; and (c) 1 fast, 1 intermediate, 1 slow (i.e., a three time-scale model).  Each of these options can create MMOs through different mechanisms.  Models with 1 fast and 2 slow variables can create MMOs through a folded node or folded saddle-node with a global return mechanism that repeatedly sends trajectories near the singularities.  Models with 1 slow and 2 fast variables can create MMOs through a delayed Hopf mechanism that also requires a global return.  MMOs in three time-scale models are reminiscent of MMOs due to a folded saddle-node type II---where one of the equilibria is a folded singularity ---although the amplitudes of the SAOs are more pronounced in this case.  

To verify our hypothesis, we have to strike a delicate balance.  The model needs to be complex enough to exhibit the desired behavior, but if it is too complex we will be unable to {\it prove} that it does so.  We know from the data shown in Figure \ref{mmoFig:tempRecord} that temperature shows relaxation behavior, rapidly oscillating between two meta-stable states.  Assuming that ice volume is strongly correlated with temperature, we view glacial cycles in the same way, i.e., oscillating between two meta-stable states.  The possibility of a bistable regime in ice volume has been discussed for decades, notably by Weertman \cite{weertman1976}, MacAyeal \cite{macayeal1979}, Oerlmans \cite{oerlemans1981}, Calov and Ganopolski \cite{calgan05}, Crucifix \cite{crucifix11}, and Abe-Ouchi {\it et al}. \cite{abeouchiEA13}.  Atmospheric carbon should also play a role in any model that describes glacial-interglacial cycles, as suggested by Saltzman and Maasch \cite{ms91} as well as Paillard and Parrenin \cite{pp04}.  We consider a physical, conceptual model that incorporates continental ice sheets, atmospheric carbon, and oceanic carbon.  Since this approach has never been used in a climate-based model, our desire is that the analysis is clear enough to replicate.  This is a major reason for our choice of such a simplistic 3D model.  Indeed, we omit time-dependent forcing such as Milankovitch cycles, leaving these effects to future work.  Even so, a minimal model is able to provide insight into key mechanisms behind the MMOs.  We include oceanic carbon as the third variable because the model was able to produce MMOs.  However, we were unable to find MMOs in other minimal models with, for example, deep ocean temperature. 

Our analysis will rely heavily on the model and ideas put forth by MacAyeal in \cite{macayeal1979} where the physical units---as well as the physical meaning of some parameters---are ambiguous.  His approach to explaining glacial cycles with a catastrophe model is similar to our MMO approach, without the benefit of 25 years of mathematical development.  Rather than using independently varying parameters as a means of generating slow dynamics, we couple MacAyeal's model with (simplified) carbon dynamics.  The main task is to obtain the ``global'' time-scale separation between the ice sheet evolution and the evolution of the carbon equations denoted by $\epsilon_1$ and $\epsilon_2$.  In general, a time-scale separation can be revealed through dimensional analysis.  The process should relate a small parameter $\epsilon_i$ to physical parameters of the dimensional model.  In applications such as neuroscience, it is often possible to get a handle on the ``smallness'' of the $\epsilon_i$ because there are accepted values or ranges for many of the physical parameters.  Unfortunately, parameters in paleoclimate models are not as constrained.  We rely on the intuition of physicists, geologists, and atmospheric scientists to determine a reasonable separation of time-scales.    

While it may be unsettling to not have a more concrete argument, the ambiguity regarding parameter values---and even the governing equations---allows more freedom.  With this in mind we take a different approach than that of others in the paleoclimate literature such as Saltzman and Maasch \cite{ms88, ms90, ms91}.  In the vast majority of climate science papers, the authors simulate models with judiciously chosen parameters.  Our approach is different in that we assume nothing about any parameters except that they are physically meaningful.  Then, through the analysis, we find conditions under which the model behaves qualitatively like the data.  The idea is not to pinpoint specific parameter values, but to find a range of possible parameters.  There are two advantages to this approach.  First, the parameter range can be used to constrain (or maybe constrain further) previous parameter estimates, which may tell us something previously unknown about the climate system.  It can be used to inform parameter choices for large simulations.  Second, a parameter range is useful to eliminate options.  That is, if the only parameter range which produces the correct qualitative behavior is entirely unreasonable, the model needs to be changed.

The outline of the paper is as follows: In section 2 we set up the model and provide relevant background from the paleoclimate literature.  Then we nondimensionalize the model and discuss assumptions on some of the parameters.  In particular, we identify our dimensionless model as a multiple time-scale problem.  We analyze this model in section 3, with a focus on finding conditions for MMOs.  We conclude with a discussion in section 4.

\section{Setting up the Model}
\subsection{The Physical Model}
We start with a model of the form
	\begin{align}
		\label{sie}
		\gamma \frac{dX_e}{dt} &= A_0 (B_0 - A) - B_1 X_e^3 + B_2 X_e \\
		\label{atmosCO2}
		\frac{dA}{dt} &= B_3 [P (X_e - X_e^*)^2 -  B_4 - A] - (L + B_5 A - B_6 H) \\
		\label{hydroCO2}
		\frac{dH}{dt} &= L + B_5 A - B_6 H.
	\end{align}
	
\begin{table} [b!] 
\begin{center}
\includegraphics[width=\textwidth]{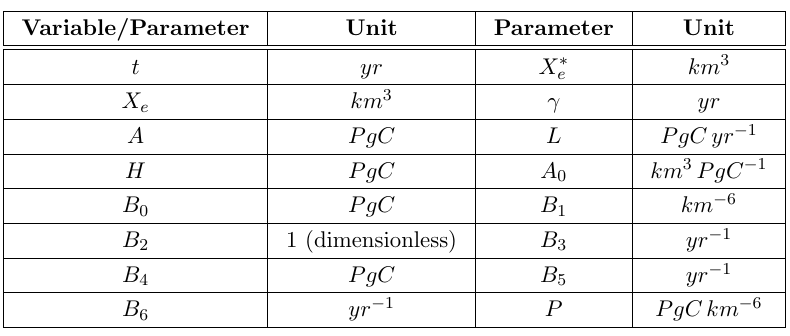}
\caption{Summary of the parameters, variables and their units.} \label{parametertable}
\end{center} 
\end{table}

$X_e$ is continental ice volume ($km^3$) offset from some mean value.  $A$ is PgC (Petagrams of Carbon) in the atmosphere, and $H$ is PgC in the mixed layer of the ocean.  Often atmospheric carbon is discussed as carbon concentration in the atmosphere in ppm (parts per million).  However when discussing land-atmosphere flux, as we will do, it makes sense to discuss carbon in terms of mass, hence the choice of PgC \cite{sigmanboyle}.  Equation \eqref{sie} is a variant of the equation used in MacAyeal's catastrophe model \cite{macayeal1979}.  The parameter $\gamma$ is a time-scale parameter that determines how quickly the ice-sheets relax to equilibrium.  In MacAyeal's formulation, $\gamma$ is actually a variable depending on $X_e$ and $t$, however it is assumed to be small and positive.  We have dropped the dependence, making $\gamma$ a true parameter.  The effect of atmospheric carbon and solar insolation is encapsulated by the term $A_0 (B_0 - A)$.  Parameters $B_1$ and $B_2$ are not given explicit, quantitative interpretations in \cite{macayeal1979}.  When $B_2$ is positive, the ice sheet dynamics may be in a bistable regime due to the cubic nature of equation \eqref{sie}.  This bistability plays an important role in demonstrating the capability for MMOs.  Our discussion of bistability differs slightly from MacAyeal's.  For our purposes, bistability in ice volume requires a separation of time-scales---corresponding to the relative evolution rates of ice sheet dynamics and the carbon pools---that will be made explicit in a dimensionless model.

Equation \eqref{atmosCO2} can be decomposed into two terms: the land-atmsophere flux,
$$B_3 [P (X - X_e^*)^2 -  B_4 - A],$$
and the ocean-atmosphere flux
$$L + B_5 A - B_6 H.$$
Notice that the ocean-atmosphere flux is balanced in equation \eqref{hydroCO2}.  That is, whatever carbon is outgassed from (or absorbed by) the ocean must be transferred to (or from) the atmosphere.  The ocean-atmosphere flux is an extremely complicated process.  First, the ocean has numerous carbon reservoirs of various sizes (e.g. the mixed layer and the deep ocean \cite{sigmanboyle}).  Changes in ocean circulation alter the amount of carbon that is pumped into the mixed layer versus what is stored in the deep ocean.  In the work of Paillard and Parrenin, the key nonlinearity in the carbon dynamics that drives the glacial-interglacial cycles occurs in the ocean-atmosphere component \cite{pp04}.  Additionally, the air-sea exchange of carbon is temperature-dependent since the solubility of $\text{CO}_2$ depends on temperature \cite{leq04}.  Acknowledging that this is a gross simplification, we assume that the carbon exchange between atmosphere and ocean follows a simple linear equation as described in \cite{siegsarm93}.  Approximating a slow nonlinear mechanism with a linear term has proved illuminating in systems with multiple time-scales.  One particular example from neuroscience is the FitzHugh-Nagumo approximation of the Hodgkin-Huxley equations \cite{fitzhugh1961,nagumo1962}.  Considering the role of neuroscience in the historical development of MMO theory, such an approximation seems natural.

The terrestrial, or land-atmosphere, flux depends on the lithosphere and the biosphere (among other things) \cite{landCO2}.  Adams and Faure \cite{adamsfaure1998}; Crucifix, Betts, and Hewitt \cite{cbh05}, K\"{o}hler and Fischer \cite{kf05}, and Lenton and Huntingford \cite{landCO2} all discuss a reduced terrestrial carbon pool at the last glacial maximum.  Our model adapts the formulation used by Lenton and Huntingford \cite{landCO2}, assuming that there is a critical ice sheet volume $X_e^*$(corresponding to a critical temperature) where carbon drawdown is most efficient.  Despite the consensus that there was a reduced terrestrial carbon pool at the last glacial maximum, the quantification of such is still debated \cite{kf05}.  The quantification debate is one reason for our hesitance to assign specific values to parameters.   

The reason there is no governing equation for terrestrial PgC is that the total carbon content of the system should be conserved.  While there can be subdivisions within them \cite{sigmanboyle}, we are considering three carbon stores: atmosphere, land, and ocean.  Since there is a conserved quantity, only the two governing equations are needed.

\subsection{The Mathematical (Dimensionless) Model}
In an effort to simplify calculations, we set 
$$X = - X_e,$$
so $X$ grows as temperature increases, allowing for an easier comparison with Figure \ref{mmoFig:tempRecord}.   Note that Similarly, we introduce
$$X_* = - X_e^*.$$
In terms of the variables $X, A, H$, the system \eqref{sie}-\eqref{hydroCO2} becomes
	\begin{align}
		\label{sie2}
		\gamma \frac{dX}{dt} &= - \gamma \frac{dX_e}{dt} = A_0 (A-B_0) - B_1 X^3 + B_2 X \\
		\label{atmos2}
		\frac{dA}{dt} &= B_3 [P (X - X_*)^2 -  B_4 - A] - (L + B_5 A - B_6 H) \\
		\label{hydro2}
		\frac{dH}{dt} &= L + B_5 A - B_6 H.
	\end{align}

Secondly, based on the observation made in Figure \ref{mmoFig:tempRecord}, the model \eqref{sie2}-\eqref{hydro2} should evolve on multiple time-scales.  Such a separation of time-scales can only be identified in a dimensionless model.  Therefore, we define the dimensionless quantities
\begin{equation*}
	\begin{array}{cccc}
		\displaystyle{x = \frac{X}{X_c}}, & \displaystyle{y = \frac{A}{A_c}}, & \displaystyle{z=\frac{H}{H_c}}, \text{ and} & \displaystyle{s = \frac{t}{t_c}}
	\end{array}
\end{equation*}  
where
\begin{align*}
		X_c &= \sqrt{ \frac{B_2}{3 B_1} }, &
		A_c &=\frac{B_2}{3 A_0} X_c ,  \\
		H_c &= \frac{B_5 A_c}{B_6}, &
		t_c &= \frac{3 \gamma}{B_2}.
\end{align*}  
Then equations \eqref{sie2}-\eqref{hydro2} become 
\begin{align}
	\label{xp}
	\epsilon \dot{x} &= y - x^3 + 3x - k \\
	\label{yp}
	\dot{y} &=  p (x-a)^2 - b - my - (\lambda + y) +z \\
	\label{zp}
	\dot{z} &= r(\lambda + y - z),
\end{align}
where the dot ($\dot{ \ \ }$) denotes $\frac{d}{ds}$.  The new dimensionless parameters relate to the physical parameters of equations \eqref{sie2}-\eqref{hydro2} in the following way:
\begin{equation*}
	\begin{array}{cccc}
		\displaystyle{ k = \frac{3 B_0 A_0}{B_2 X_c} }, & \displaystyle {p =  \frac{3 P B_3 A_0 X_c}{B_2 B_5} } & \displaystyle {a = \frac{X_*}{X_c} }, & \displaystyle{ b = \frac{B_3 B_4}{B_5 A_c} }   \\  \noalign{\bigskip}
		\displaystyle{m = \frac{B_3}{B_5} }, & \displaystyle{ \lambda = \frac{L}{B_5 A_c} }, & \displaystyle { r = \frac{B_6}{B_5} }, \text{ and} & \displaystyle{ \epsilon = \frac{3 \gamma B_5}{B_2} }.
	\end{array}
\end{equation*}
Any time-scale separation is determined by $\epsilon_1 = \epsilon$ and $\epsilon_2 = \epsilon r$.  As mentioned earlier, parameter values in paleoclimate problems are the subject of some debate.  In accordance with our observation based on Figure \ref{mmoFig:tempRecord}, we assume that temperature evolves on a faster time-scale than carbon, implying $0< \epsilon \ll 1$.  If $r = \mathcal{O}(1)$ we have 1 fast and 2 slow variables, and if $r \ll 1$ we are in the three time-scale case.  Depending on which parameters hold the key to having $0 < \epsilon \ll 1$, other parameters (e.g. $p$, $b$, or $\lambda$) may be small as well.  Again, our approach is to assume as little as possible about the parameters, so we will keep this in mind as we perform the analysis.

\begin{remark}
The dimensionless form of the model is a variant of the Koper model, an electrochemical model that is known to exhibit MMOs \cite{koper, kuehn11}.  Many other models in chemistry and neuroscience also demonstrate MMOs (e.g., modified Hodgkin-Huxley equations \cite{rw08}).  Indeed, many mechanisms in other areas such as mass balance in chemical reactions or gated ion channels in neural models behave similarly to certain climate mechanisms such as conservation of mass or exchange of carbon dioxide across the ocean-atmosphere surface.  
\end{remark}

The geometric theory for analyzing dynamical systems with multiple time scales---known as {\it Fenichel theory} or {\it geometric singular perturbation theory} (GSPT) \cite{fenichel79, gsp}---has provided powerful tools for studying singular perturbation problems such as system \eqref{xp}-\eqref{zp}.  Together, GSPT and {\it blow-up techniques} \cite{dumort96, ksRO, szmolwechs} provide rigorous results on global behavior such as relaxation oscillations \cite{sw04}.  Figure \ref{mmoFig:examples1} depicts another type of complex oscillator behavior called mixed-mode oscillations (MMOs).  Neurophysiological experiments are known to produce similar patterns \cite{amir02, dickson1998, gutfreund1995, khosrovani07, delnegro02}.  Recently, these complicated oscillations have been explained using canard theory \cite{wec07} in conjunction with an appropriate global return mechanism by exploiting the multiple time-scale nature of the underlying models \cite{benoit1981, benoit1983, brons, guckenheimer08, milik1998, szmolwechs, mw05, wec12}.  This is now one widely accepted explanation for MMOs; see, e.g., \cite{brons08, ddw09, mmosurvey}.

\begin{figure}
        \centering
           \begin{subfigure}[t]{0.4\textwidth}
                \includegraphics[width=\textwidth]{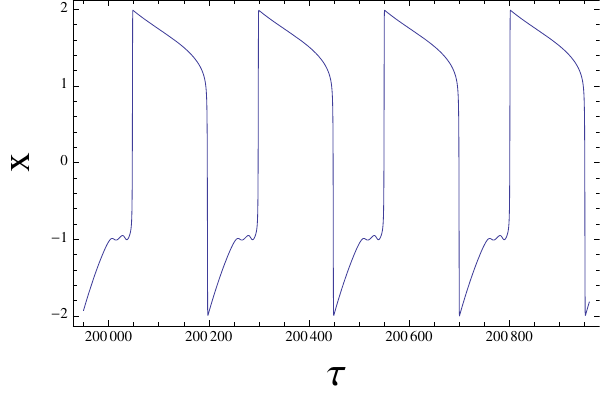}
                \caption{Stable MMO orbit with parameters $\epsilon = 0.01, \ a=0.8, \ p =3, \ b=2.1, \ k = 4, \ r=1, \ m=1,$ and $\lambda=1$.}
                \label{mmoFig:seriesE2}
        \end{subfigure}
       \qquad
       \begin{subfigure}[t]{0.4\textwidth}
                \includegraphics[width=\textwidth]{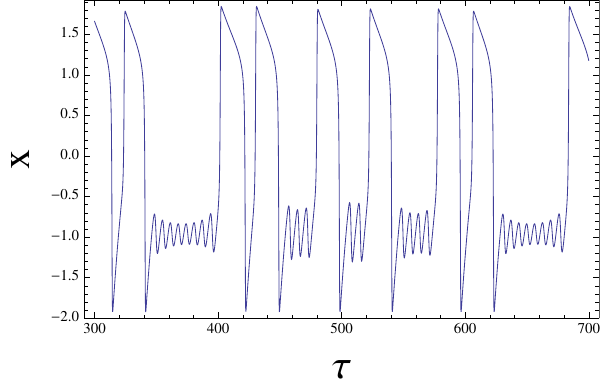}
                \caption{Stable MMO orbit with parameters $\epsilon = 0.1, \ a=0.8, \ p=3, \ b = 2, \ k =4, \ r = 0.05, \ m=1,$ and $\lambda = 1$.}
                \label{mmoFig:example3scales}
        \end{subfigure}
        \caption{Example of MMO patterns generated from the model \eqref{xp}-\eqref{zp}.}
        \label{mmoFig:examples1}
\end{figure}

\section{Analyzing the System}
In this section we will analyze the system \eqref{xp}-\eqref{zp}.  We assume that the system is singularly perturbed with singular perturbation parameter $\epsilon$.   We also assume that $r =\mathcal{O}(\epsilon^n)$ where $n = 0$ or $n=1$ (although fractional powers may be acceptable as well).  Hence we are using a 2 slow/1 fast approach that allows for the case where $r = \mathcal{O}(\epsilon).$  We will comment on the case where $r$ is small when appropriate.

The following quantities will appear often in our calculations, so we define
\begin{align*}
	h(x) &= x^3-3x+k,\\
	f(x)&= p(x-a)^2-b,
\end{align*}
as well as
$$ F(x,y) = y - h(x).$$

\subsection{The Layer Problem}
To begin the analysis, we rescale the time variable $s$ by $\epsilon^{-1}$ to obtain the system
\begin{align}
	\label{xpF}
	x' &= y- h(x) = y - x^3 + 3x - k \\
	\label{ypF}
	y' &= \epsilon [f(x) - m y - (\lambda + y) + z] = \epsilon [p (x-a)^2 - b - my - (\lambda + y) +z] \\
	\label{zpF}
	z' &= \epsilon r(\lambda + y - z),
\end{align}
where the prime (') denotes $d/d\tau$, and $\tau = \epsilon^{-1} s$ is the fast time-scale (while $s$ is the slow time-scale).  As long as $\epsilon > 0$, the new system \eqref{xpF}-\eqref{zpF} is equivalent to \eqref{xp}-\eqref{zp} in the sense that the paths of trajectories are unchanged---they are merely traced with different speeds.  However, in the singular limit (i.e. as $\epsilon \rightarrow 0$) the systems are different.

When $\epsilon = 0$, the system \eqref{xpF}-\eqref{zpF} becomes
\begin{align*}
	x' &= F(x,y) \\
	y' &= 0\\
	z' &= 0,
\end{align*}
which is called the {\it layer problem.}  Notice that the dynamics in the $y$ and $z$ directions are trivial.  
The {\it critical manifold},
$$M_0 = \{F(x,y) = 0\} = \{ y = h(x) \}, $$
is the set of critical points of the layer problem.  $M_0$ is attracting (resp. repelling) whenever $F_x < 0$ (resp. $F_x > 0$), which corresponds to the $x$-values where the cubic $h(x)$ is increasing (resp. decreasing).  A simple calculation shows $h'(x) = 0$ when $x = \pm 1$, so $M_0$ is attracting on the outer branches where $|x| > 1$, repelling on the middle branch where $|x|<1$ and folded at $x = \pm 1$.  To make this more explicit, $M_0$ is `S'-shaped with two attracting branches
$$M^{\pm}_A = \{ \pm x > 1 \} $$
and a repelling branch
$$M_R = \{ -1 < x < 1 \}.$$  The attracting and repelling branches are separated by the folds
$$L^{\pm} = \{ x = \pm 1 \}.$$
At the folds $L^{\pm}$, the critical manifold is degenerate and the basic GSPT theory for normally hyperbolic critical manifolds breaks down.  As is so often the case, the scientifically and mathematically interesting behavior arises where the standard theory does not apply.  In our case, the folds allow for more complicated dynamics such as relaxation oscillations or MMOs.  

\subsection{The Reduced Problem}
The layer problem, which describes the fast dynamics off the critical manifold, was obtained by considering the $\epsilon =0$ limit of equations \eqref{xpF}-\eqref{zpF}.  The dynamics on the critical manifold, or slow dynamics, are obtained by looking at the system \eqref{xp}-\eqref{zp} as $\epsilon \rightarrow 0.$  In the singular limit, the system becomes
\begin{align}
	\label{redM0}
	0 &= y - h(x)  \\
	\label{redY}
	\dot{y} &= f(x) - m y- (\lambda + y) +z   \\
	\label{redZ}
	\dot{z} &= r(\lambda + y - z).
\end{align}

The system \eqref{redM0}-\eqref{redZ} is called the {\it reduced problem}.  It is a differential algebraic system---i.e., a differential equation on the manifold $y = h(x,z)=h(x)$ which is a graph over the coordinate chart $(x,z)$.  As usual, we study manifolds in charts (i.e., in local coordinates), and here we have a single coordinate chart $(x,z)$ where we can study the whole reduced flow. This is done by differentiating the algebraic condition in \eqref{redM0}, and substituting it for the $\dot{y}$ equation \eqref{redY}.  Doing so produces
\begin{equation*}
		\begin{array}{rcl}
			-F_x \dot{x} &=& F_y \dot{y} + F_z \dot{z} = F_y \dot{y} \\
			\dot{z} &=& r ( \lambda + y - z ).
		\end{array}
	\end{equation*}
Substitution provides
	\begin{equation}
		\label{mmo2}
		\begin{array}{rcl}
			h'(x) \dot{x} &=& f(x)-(m+1) h(x) - \lambda + z \\
			\dot{z}& = &r ( \lambda + h(x) - z ),
		\end{array}
	\end{equation}
and we have obtained an expression for the reduced problem \eqref{redM0}-\eqref{redZ} as a system \eqref{mmo2} in the coordinate chart $(x,z)$. 

Notice that $h'(\pm 1) = 0.$  Points along the set $\{x=\pm1 \}$ are called fold points because they correspond to extrema of the critical manifold where it appears folded (denoted by red lines in Figure \ref{mmoFig:fullOrbit}).   Additionally, the system \eqref{mmo2} is singular at the folds because the coefficient of $\dot{x}$ is 0.  System \eqref{mmo2} has three different types of singularities:
\begin{itemize}
	\item ordinary singularities---these are equilibria of the full system \eqref{xp}-\eqref{zp} 
	\item regular fold points---also known as jump points, and
	\item folded singularities---isolated points along $L^{\pm}$ where the reduced flow changes orientation.
\end{itemize}

\begin{figure}
        \centering
           \begin{subfigure}[t]{\textwidth}
           	\centering
                \includegraphics[height=0.4\textheight]{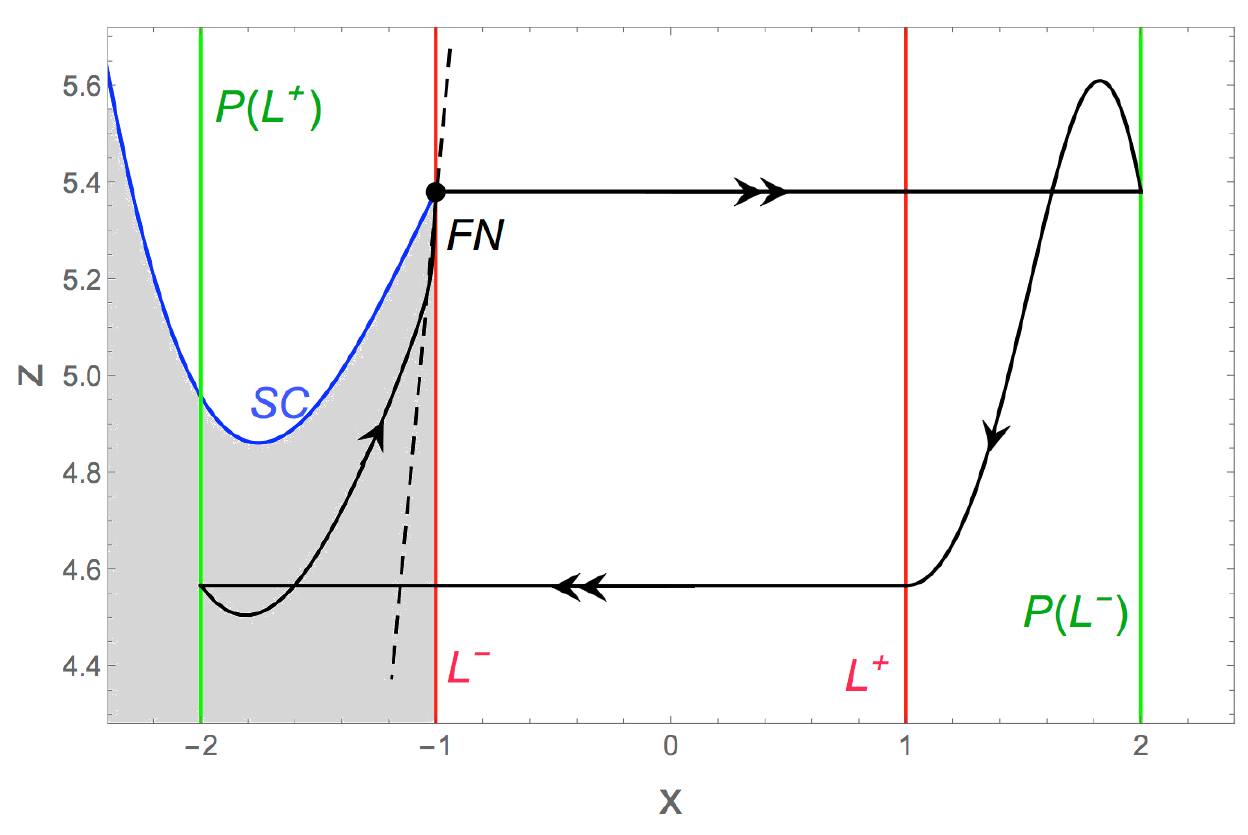}
                \caption{Projection of the singular orbit onto the critical manifold, with projections of the fold lines.  The funnel is the shaded region below the strong canard $\gamma_s$ (denoted SC).  The dashed line indicates the weak eigendirection along which trajectories inside the funnel approach FN.}
                \label{mmoFig:projectedOrbit}
        \end{subfigure}
       \\ 
       \begin{subfigure}[t]{\textwidth}
       		\centering
                \includegraphics[height=0.4\textheight]{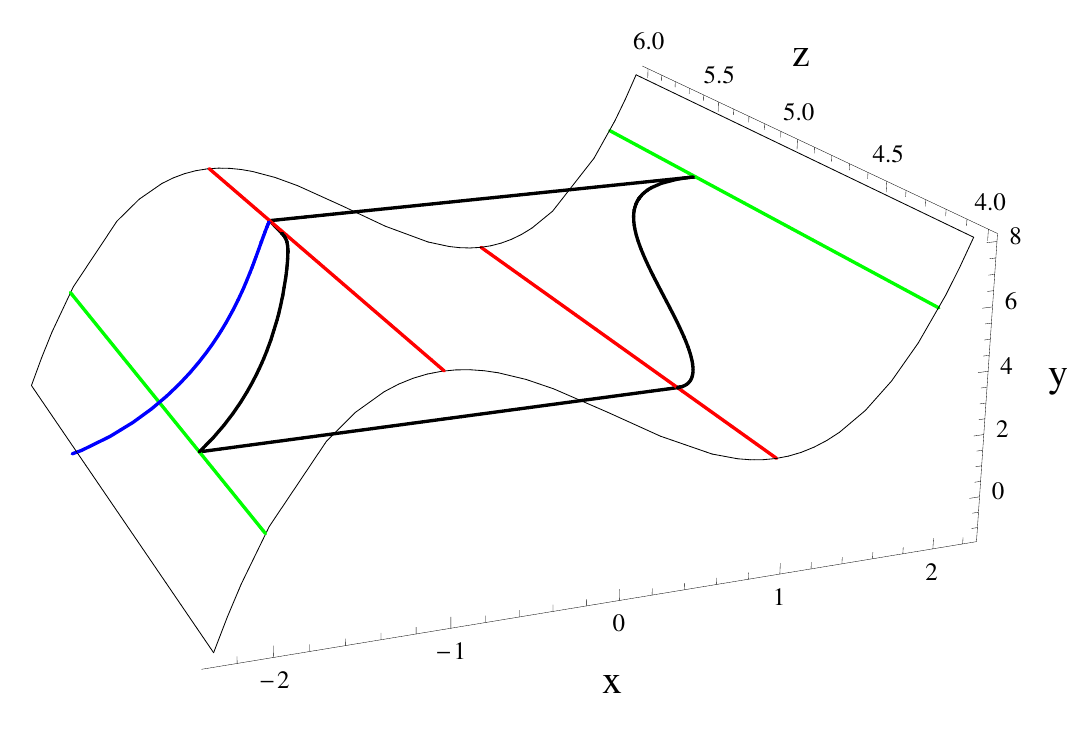}
                \caption{Singular orbit $\Gamma$ in the full 3D phase space.  Colored lines correspond to those in (a).}
                \label{mmoFig:fullOrbit}
        \end{subfigure}
        \caption{Example of a singular periodic orbit $\Gamma$ for $a=0.8$, $p=3$, $b=2.1$, $k=4$, $r=1$, $m=1$, and $\lambda=1$. Fold lines ($L^\pm$) and their projections ($P(L^\pm)$) can be computed explicitly.  Additionally, fast trajectories (denoted by double arrows) can be computed explicitly in the singular limit.  Slow trajectories (denoted by single arrows) are computed and plotted numerically.  Note that the strong canard (SC) is also a trajectory of the reduced problem and is found using a classic ``shooting method'' as implemented in, e.g. XPPAUT.  Here we have plotted a numerical approximation of the strong canard.}
        \label{mmoFig:singularOrbit}
\end{figure}

We can desingularize the system by rescaling the time variable $s$ by a factor of $h'(x)$, and we obtain
	\begin{equation}
		\label{desing}
		\begin{array}{l}
			\dot{x} = f(x)-(m+1) h(x) - \lambda + z \\
			\dot{z} = r h'(x) ( \lambda + h(x) - z ).
		\end{array}
	\end{equation} 
The rescaling reverses trajectories on $M_R$ because this is precisely the set where $h'(x) < 0$, and hence, the time variable is scaled by a negative factor.  However, the benefit of being able to define dynamics on the whole critical manifold---including the folds---outweighs the cost.  The folds themselves are transformed from singular sets where the dynamics were undefined into $z$ nullclines.  From system \eqref{desing}, it is now easy to classify the different singularities of the reduced problem: 
	\begin{itemize}
		\item Ordinary singularities occur where $f(x)-(m+1) h(x) - \lambda + z = 0$ and $ \lambda + h(x) - z = 0$.  That is, they are equilibria where $h'(x) \neq 0$.  
		\item Regular fold points are fold points that are not equilibria of \eqref{desing}. That is, $h'(x) = 0$, but $\dot{x} \neq 0.$  The red lines in Figures \ref{mmoFig:projectedOrbit}, \ref{mmoFig:zStar}, and \ref{mmoFig:m0pp} (with the exception of the points labeled `FN', `$Z_-$', and `$Z_+$') denote regular fold points.
		\item Folded singularities are equilibria of \eqref{desing} where $\dot{z}=0$ as a result of the rescaling.  That is, folded singularities occur where $h'(x) =0$, $\lambda + h(x) - z \neq 0$, and $\dot{x}=0$.  The points labeled `FN' in Figures \ref{mmoFig:projectedOrbit} and \ref{mmoFig:zStar} mark folded singularities.  Additionally, the points labeled `$Z_-$' and `$Z_+$' in Figure \ref{mmoFig:m0pp} are folded singularities.
	\end{itemize}
Similar to ordinary singularities (or ``standard'' equilibria), folded singularities can be classified by the eigenvalues of Jacobian evaluated at the folded singularity.  For example, a folded singularity with real, negative eigenvalues is a stable folded node, or if the Jacobian has real eigenvalues with opposite sign, it is a folded saddle.  The `FN' in Figures \ref{mmoFig:projectedOrbit} and \ref{mmoFig:zStar} indicate that the folded singularities are folded nodes.

\subsection{Canard Induced MMOs}

Folded nodes (as well as folded saddle-nodes) can produce MMOs with a suitable global return mechanism \cite{brons}.  A node (in a 2D system) has two real eigenvalues of the same sign, a weak eigenvalue $\mu_w$ and a strong eigenvalue $\mu_s$ such that $|\mu_w| < |\mu _ s|$.  Each eigenvalue corresponds to a trajectory that approaches the folded node tangent to the corresponding eigenvector.  The geometry of these special trajectories plays a key role in the return mechanism.  We denote the strong stable trajectory (corresponding to $\mu_s$) by $\gamma_s$, and similarly define the weak stable trajectory  $\gamma_w.$  

Since $\gamma_s$ is a trajectory of the slow dynamics, uniqueness of solutions prevents other slow trajectories from crossing it.  Thus, $\gamma_s$ partitions $M_A^-$ into two regions of trajectories.  In particular, $\gamma_s$ separates the trajectories that cross the fold before reaching the node from those that reach the fold at the node.  Aside from $\gamma_s$, the trajectories approaching a stable node stack up along the weak stable trajectory.  So, we see that trajectories in the region containing $\gamma_w$ will reach the fold at the node.  Trajectories on the opposite side of $\gamma_s$ will reach the fold first.
The region between the strong stable trajectory $\gamma_s$ and the fold $L^-$ that contains the weak stable trajectory $\gamma_w$ is called the {\it singular funnel} (denoted by the shaded region in Figures \ref{mmoFig:projectedOrbit} and \ref{mmoFig:zStar}), and we often refer to $\gamma_s$ as the boundary of the funnel. Since $\gamma_s$ is also called the strong canard, it is labeled 'SC' in Figures \ref{mmoFig:projectedOrbit} and \ref{mmoFig:zStar}.  All trajectories in this region will reach the fold at the folded node, while all trajectories on the other side of $\gamma_s$ cross the fold without reaching the node \cite{wec07}.

We establish a global return mechanism by constructing a singular periodic orbit $\Gamma$, consisting of heteroclinic orbits of the layer problem and a segment on each of those stable branches $M^{\pm}_A$.  The heteroclinic orbits of the layer problem take  trajectories from a fold $L^{\pm}$ to its projection $P(L^\pm)$ on the opposite stable branch.  An example of a singular periodic orbit $\Gamma$ is shown in Figure \ref{mmoFig:singularOrbit}.  Assuming there is a folded node on $L^-$ (without loss of generality), we can construct $\Gamma$ by following the fast fiber from the node to the stable branch $M_A^+$.  From there, the trajectory follows the slow flow on $M_A^+$ as described by \eqref{desing} until it reaches the fold $L^+$.  If it reaches $L^+$ at a jump point, we follow the fast fiber back to $M_A^-$.  We want the landing point on $M_A^-$ to be in the  funnel, since any singular orbit from the folded node that returns to the funnel will necessarily be a singular periodic orbit.

The eigenvalues are important for another reason as well.  The ratio of the eigenvalues,
$$\mu =  \frac{\mu_w}{\mu_s} < 1,$$
determines the number of small-amplitude oscillations in the MMO signature.  This is made explicit in the following theorem due to Br{\o}ns {\it et al} \cite{brons} that provides conditions under which a system has a stable MMO orbit.

\begin{theorem}
\label{thm:assumptions}
Suppose that the following assumptions hold in a fast/slow system,
	\begin{enumerate}[({A}1)]
		\item $0 < \epsilon \ll 1$ is sufficiently small with $\epsilon^{1/2} \ll \mu$
		\item the critical manifold is `S'-shaped, i.e. $M_0 = M_A^- \cup L^- \cup M_R \cup L^+ \cup M_A^+$, 
		\item there is a (stable) folded node $N$ on (without loss of generality) $L^-$, 
		\item there is a singular periodic orbit $\Gamma$ such that $\Gamma \cap M_A^-$ lies in the interior of the singular funnel to $N$, and
		\item $\Gamma$ crosses $L_{\pm}$ transversally.
	\end{enumerate}
	Then there exists a stable periodic orbit of MMO type $1^s$, where
	\begin{equation}
		\label{s}
		s = \left[ \frac{(1+ \mu)}{2\mu} \right],
	\end{equation}
	and the right-hand side of \eqref{s} denotes the the greatest integer less than $(1+ \mu)/(2\mu)$.
\end{theorem}

In \cite{kw2010}, Krupa and Wechselberger show that the folded node theory still applies in the parameter regime where $\mu = \mathcal{O}(\epsilon^{1/2})$ if the global return mechanism is still intact (i.e $\Gamma \cap M_A^-$ lies in the interior of the singular funnel).  Note that in this parameter regime, the MMO signature can be more complicated.  Figure \ref{mmoFig:threeScales} depicts a few of the more interesting MMO patterns generated by \eqref{xp}-\eqref{zp} when $\mu = \mathcal{O}(\epsilon^{1/2})$.  

The remainder of this section will focus on finding conditions on the parameters of equations \eqref{xp}-\eqref{zp} so that the system satisfies $(A1)$-$(A5).$  Since $\mu$ is calculated in the singular limit, we can always choose $\epsilon$ small enough to satisfy condition $(A1)$.  Also, we have already discussed the `S'-shape of the critical manifold, demonstrating that condition $(A2)$ is satisfied.  The next task will be find conditions so that equations \eqref{xp}-\eqref{zp} have a folded node singularity.

\begin{figure}
        \centering
        \begin{subfigure}[b]{0.45\textwidth}
        \centering
                \includegraphics[width=\textwidth]{temp_timeseriesE2.pdf}
		\caption{Time series for $\epsilon = 0.01$.}
        \end{subfigure}%
        ~ 
        \begin{subfigure}[b]{0.45\textwidth}
               \includegraphics[width=\textwidth]{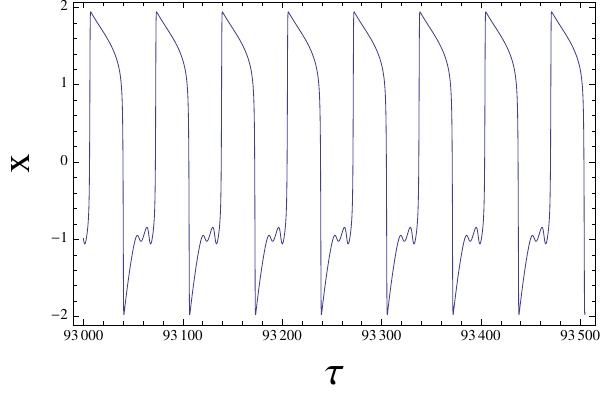}
		\caption{Time series for $\epsilon = 0.05$.}
        \end{subfigure}
        \caption{Time series for different values of $\epsilon$ when $a=0.8, \ p=3, \ b = 2.1, \ k = 4, \ r = 1, \ m= 1,$ and $\lambda=1$.}
        \label{mmoFig:timeseries}
\end{figure}

\subsection{Folded Node Conditions}	
The data in Figure \ref{mmoFig:tempRecord} show small amplitude oscillations occurring at low temperatures, so we seek parameters for which \eqref{mmo2} has a stable folded node along the lower fold $L^-$.  
\begin{lemma}
	\label{thm:node}
Define 
\begin{equation}
	\label{delta}
	\delta = f(-1) - m h(-1) = p (a+1)^2 - b - m(k +2).
\end{equation}  
Assume the parameters of the system \eqref{xp}-\eqref{zp} satisfy
	\begin{enumerate}[(a)]
		\item $p > 0$,
		\item $a > -1$,
		\item $\delta > 0,$ and
		\item $p^2 (a+1)^2 -6 r \delta > 0.$
	\end{enumerate}
Then there is a folded node at $(-1,z_-)$ where $$z_- = 2+ k + \lambda - \delta.$$
\end{lemma}

\begin{proof}
The linearization of \eqref{desing} at any fixed point $(x_0,z_0)$ is
\begin{equation}
\label{lin}
J(x_0,z_0) = \left( \begin{array}{cc}
f'(x_0) - (m+1) h'(x_0) & 1 \\
r  [(h'(x_0))^2 + h''(x_0)   ( \lambda + h(x_0) - z_0)] & -r \ h'(x_0)
\end{array} \right).
\end{equation}
There is a folded singularity at $(-1, z_-)$ where $$z_- = (m+1) h(-1) + \lambda - f(-1).$$  Since $h'(-1)=0$, we have the linearization
\begin{equation}
	\label{jacobian}
		J(-1,z_-) = \left( \begin{array}{cc}
		f'(-1) & 1 \\
		-6r [f(-1) - m h(-1)]& 0
	\end{array} \right).
\end{equation}
For $(-1,z_-)$ to be a stable folded node, $J(-1,z_-)$ must satisfy three conditions:
\begin{enumerate}
	\item $\text{Tr}(J(-1,z_-))< 0$,
	\item $\det(J(-1,z_-)) > 0,$ and
	\item $[\text{Tr}(J(-1,z_-))]^2-4\det(J(-1,z_-)) > 0.$
\end{enumerate}
The requirement on the trace implies that $f'(-1) < 0,$ or $p (-1-a) < 0.$  Assuming $p > 0$, we arrive at condition (b) $a > -1$.  The requirement on the determinant gives us $6r[f(-1) - m h (-1)] > 0$.  Since $r > 0$, we will have $\det(J(-1,z_-)) > 0$ whenever $\delta = f(-1) - m h (-1)>0.$  That is precisely condition (c).  Conditions (a)-(c) are enough to guarantee that the folded equilibrium is stable, but they do not distinguish between stable a stable node or a stable focus.  This is determined by the discriminant condition, which is satisfied if 
$$p^2 (a+1)^2 - 6r [ f(-1) - m h(-1) ] = p^2 (a+1)^2 - 6r \delta > 0.$$ 
\end{proof}

Note that $| \delta |$ is precisely the distance along the fold from the node to the intersection of the true $z$ nullcline with the fold at $x = -1$.  If $\delta > 0$, which is required by condition (b), then the node lies under the $z$ nullcline on $M_0$.  That is, if $z_n$ is the intersection of the $z$ nullcline with the fold (i.e., $z_n = h(-1) + \lambda$), then $z_n > z_-$ with $z_n =  z_- + \delta$.  As we will see, the parameter $r$ will not appear in the remaining calculations.  Thus we strive to find conditions on $\delta$, $a$, and $p$.  Choosing values that satisfy those conditions, (c) then provides an upper bound on $r$.

\begin{remark}
 In each of the limiting cases $r \rightarrow 0$ and $\delta \rightarrow 0$, the Jacobian \eqref{jacobian} will have a zero eigenvalue and the system will have a folded saddle-node of type II.  Near the $r=0$ limit we are in the three time-scale case with a global three time-scale separation.  Near the $\delta = 0$ limit, we have a local three time-scale split at the folded singularity.  In either case, the ratio of eigenvalues $\mu$ will be small, so near the saddle-node limit, we use the theory for $\mu = \mathcal{O}(\epsilon^{1/2})$.
 \end{remark}

Having found conditions for a folded node, it remains to be shown that these conditions are consistent with a return mechanism satisfying $(A4)$ and $(A5)$ from Theorem \ref{thm:assumptions}.  As indicated by $(A4)$, the singular funnel is a vital component of the global return mechanism.  Typically, the functionality of the return mechanism is demonstrated numerically \cite{kuehn11, rubinwechs}.  This is done by choosing a set of reasonable parameters and then varying one parameter until the MMO orbit disappears.  For example we set parameters to have the following values: $p=3, \ b=2.1, \ k =4, \ r=1, \ m=1, \ \lambda=1,$ and we vary $a$.  If $a \approx 0.643$, we have that $\delta \approx 0$ and we are in the SN-II case.  From this value we increase $a$ to 0.8 to obtain the singular orbit in Figure \ref{mmoFig:singularOrbit}.  Figure \ref{mmoFig:timeseries} shows stable MMO time series for these parameters away from the singular limit.  Continuing to increase $a$, we see that when $a \approx 0.823$, the singular periodic orbit lands on the strong canard.  For $a \geq 0.824$, there is no MMO orbit since the return mechanism no longer sends trajectories into the funnel. 

In the following subsections, we use approximations to obtain analytical results and prove our main result.  The strategy is to linearly approximate strong canard and show that the approximation lies within the funnel.  Then we find conditions under which the return mechanism lands in the approximated funnel region.

 \subsection{Estimate of the Funnel}
 Assuming the node conditions (a)-(d) from Lemma \ref{thm:node} are met, the folded singularity will have a strong stable eigenvalue (eigenvector) and a weak stable eigenvalue (eigenvector).  Let $\mu_{s,w}$ be the eigenvalues, where $s$ and $w$ denote strong and weak, respectively.  Then
$$\mu_s < \mu_w < 0.$$
Also let $(x_{s,w},z_{s,w})$ denote the corresponding eigenvector.  A simple computation shows the slope of the eigenvector
$$ m_i = \frac{z_i}{x_i} = \frac{-6r \delta }{\mu_i} > 0,$$
where $i$ can be either $s$ or $w$.  Then we have the following relationships 
$$0 < m_s < m_w < -f'(-1),$$
where $-f'(-1)$ is the slope of the $x$ nullcline at the node.  Recall that the singular funnel is the region bounded by the fold $L^-$ and the strong canard $\gamma_s$ (the trajectory that approaches the node with slope $m_s$) that contains the weak canard.  In our case, locally near the folded node, the funnel will lie below the strong canard.   Following $\gamma_s$ away from the node in reverse time, we see that if $\gamma_s$ intersects the $x$ nullcline, it will turn down and to the right until it intersects the fold $L^-$.  We want to avoid this situation since it effectively precludes a global return mechanism.  However, if $\gamma_s$ intersects the $z$-nullcline, then it will continue up and to the left in reverse time as in Figure \ref{mmoFig:zStar}.  The following lemma provides conditions under which $\gamma_s$ lies entirely above its tangent line at the node, allowing us use a linear approximation to find a lower bound for the intersection of $\gamma_s$ with $P(L^+)$.  

\begin{figure}[t]
	\begin{center}
	\includegraphics[width=0.7\textwidth]{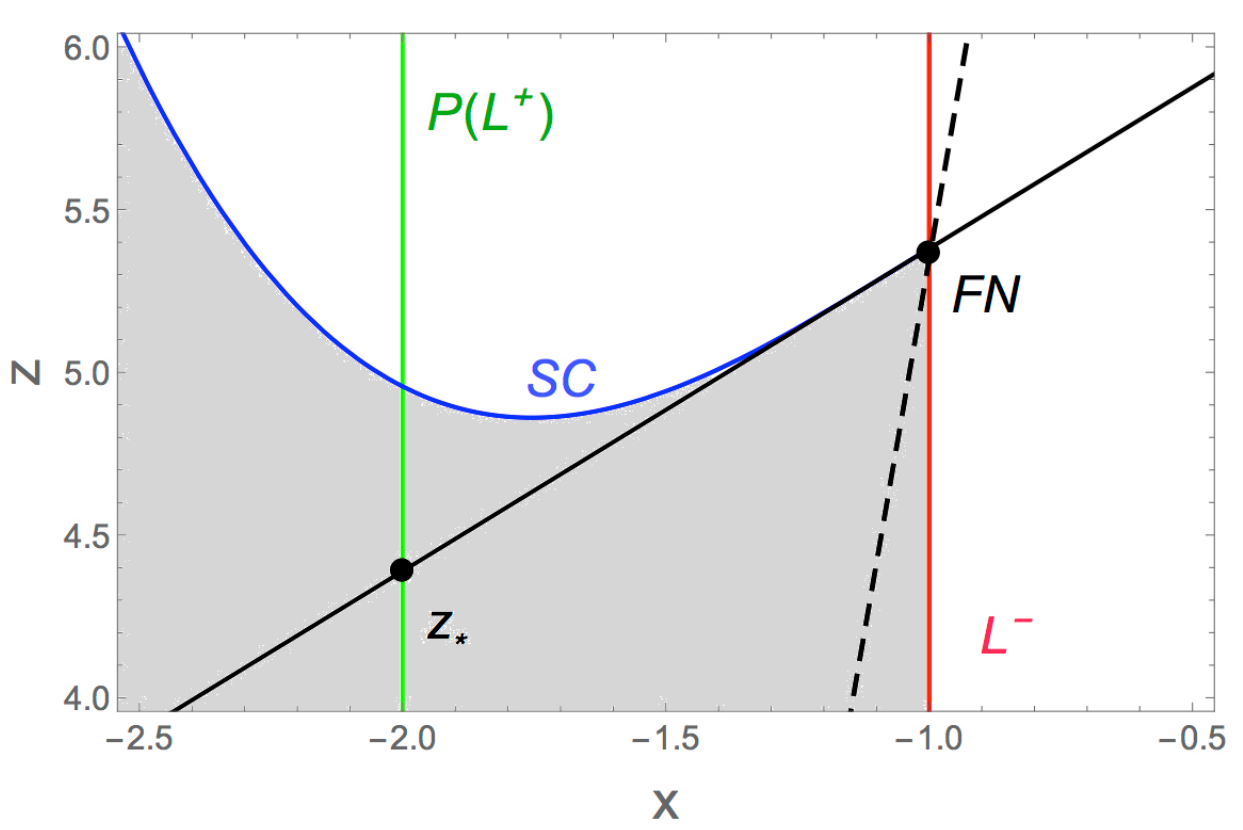}
	\caption{A lower bound for the edge of the funnel.  The fold line $L^-$ is again in red.  $P(L^+)$, drawn in green, is the set where fast trajectories leaving $L^+$ land on $M_A^-$.  The solid black line containing $z_*$ is the strong eigendirection, and the dashed line is the weak eigendirection along which trajectories inside the funnel approach FN.}
	\label{mmoFig:zStar}
	\end{center}
\end{figure} 

\begin{lemma}
	\label{thm:funnel}
		Let equations \eqref{xp}-\eqref{zp} satisfy the conditions (a)-(d) of Lemma \ref{thm:node}.  Furthermore assume
		\begin{enumerate}[(a)]
		\setcounter{enumi}{4}
			\item $\displaystyle{\frac{2 p^2 (a+1)^2}{\delta} + 2pa-6(m+1) < 0}$, and
			\item $p(a+1) > 2.$
		\end{enumerate}
		Let $m_s$ denote the slope of the strong eigenvector to the node.  That is
		$$ m_s = \frac{6 r \delta}{ - \mu_s}=\frac{6 r \delta}{ | \mu_s |}.$$
		Then the strong canard $\gamma_s$ is tangent to the line $ z = m_s (x+1) + z_-$ at $x = -1$, and lies above the line for $x < -1$.
\end{lemma}

The method of proof is to show that $\gamma_s$, thought of as $z = \gamma_s (x)$, is concave up at the node $(-1,z_-)$.  This shows that $\gamma_s$ lies above the line 
$$z  - m_s (x+1) = z_-$$
near the folded node.  We will then consider the direction of the vector field along the line to show that $\gamma_s$ remains above the line.

\begin{proof}
The $z$ coordinate of the strong canard tends to $z_-$ as $x \rightarrow -1$, however since the point $(-1,z_-)$ is a node, there are many trajectories that do so.  The strong canard can be characterized as the trajectory whose slope tends to $m_s$ as $x \rightarrow -1.$  That is 
$$\lim_{x \rightarrow -1} \frac{dz}{dx} = \frac{6 r \delta}{ | \mu_s |}.$$
The concavity of the strong canard determines whether it approaches its tangent line from above or below.  We begin with the first derivative,
$$\frac{dz}{dx} =\frac{r h'(x) (\lambda+h(x) -z)}{f(x)-(m+1)h(x) - \lambda + z}.$$  
To assist us in the calculations, we define
\begin{align*}
	\eta(x) &= f(x)-(m+1)h(x)-\lambda+z \\
	\phi(x) &= \lambda + h(x) -z
\end{align*}
noting that 
\begin{equation*}
	\begin{array}{lll}
		\phi(-1) = \delta, 
			& \phi'(-1) = \displaystyle \frac{- 6 r \delta}{ | \mu_s |},
			& \eta(-1) =0, \\  [\bigskipamount]
		h'(-1) =0, 
			& h''(-1) =-6,
			& h'''(x)  =6
	\end{array}
\end{equation*}
and
$$\lim_{x \rightarrow -1} \frac{h'(x)}{\eta(x)} = \frac{6}{|\mu_s|}.$$

Now, we use the quotient rule to obtain:
$$ \frac{d^2 z}{dx^2} = \frac{r}{(\eta(x))^2} \left[ \eta(x) \left( h'(x) \left(h'(x) - \frac{dz}{dx} \right) + \phi(x) h''(x) \right)-h'(x) \phi(x) \eta'(x) \right].$$
In particular, we are interested in 
$$ \mathcal{L} = \lim_{ x \rightarrow -1} \frac{d^2 z}{dx^2}.$$
Using L'Hopital's rule, we see
\begin{align*}
	 \mathcal{L} &=\lim_{x \rightarrow -1} \left[ \frac{r}{2 \eta(x) \eta'(x)} \right. \\
	 & \cdot \left[  \eta(x) \left( h'(x) \left( h''(x) - \frac{d^2 z}{dx^2} \right) +  h''(x) \left( h'(x) - \frac{dz}{dx} \right) 
			+ \phi(x) h'''(x)  + h''(x) \phi(x) \right) \right. \\
		& + \eta'(x) \left( h'(x) \left(h'(x)- \frac{dz}{dx} \right) + \phi(x) h''(x) \right) - \eta'(x) \phi(x) h''(x) \\
		& \left. \left. - h'(x) \left( \phi(x) \left(f''(x) - (m+1) h''(x) + \frac{dz}{dx} \right) + \phi'(x) \eta'(x) \right) \right] \right], 
\end{align*}
which simplifies to
$$ \mathcal{L} = \frac{3r}{\eta'(-1)}\left( \frac{12r\delta}{| \mu_s |} + \delta \right) - \frac{3r}{| \mu_s|} \left(\frac{6r\delta}{|\mu_s|} \right) -\frac{3r}{| \mu_s|} \left( \frac{\delta (2 p +6(m+1)+ \mathcal{L} )}{\eta'(-1)} - \frac{6r\delta}{| \mu_s| } \right). $$
Simplifying further and gathering the $\mathcal{L}$ terms on one side gives
\begin{equation}
	\label{concavityL}
	\frac{\eta'(-1) | \mu_s | + 3 r \delta}{3 r \delta} \mathcal{L} = 12 r + | \mu_s| -2p - 6 (m+1).
\end{equation}
Using the node conditions---specifically the bound on $r$ from condition (d)---we can show the coefficient of $\mathcal{L}$ is negative since
\begin{align*}
	\eta'(-1) | \mu_s | + 3 r \delta &= \left( -2p(a+1) + \frac{6r\delta}{| \mu_s | } \right) \mu_s + 3 r \delta \\
			&= 9 r \delta - 2p (a+1) | \mu_s | \\
			&< \frac{3}{2} p^2 (a+1)^2 - 2 p^2 (a+1)^2 - 2 p (a+1) \sqrt{p^2 (a+1)^2 - 6 r \delta} \\
			& < 0.
\end{align*}
Since we are looking for a lower bound on the edge of the singular funnel, we want the strong canard to lie above its tangent line at the node.  So, we want $\mathcal{L} > 0$, which happens when the right-hand side of \eqref{concavityL} is negative.  That is, we want
\begin{equation}
\label{concavity2}
 12 r + | \mu_s| - 2p - 6(m+1) < 0.
 \end{equation}
Using the bound for $r$ again as well as the estimate $ | \mu_s| < 2 p ( a+1)$, we see that \eqref{concavity2} will be true if
\begin{equation}
\label{concavityFin}
\frac{2 p^2 (a+1)^2}{\delta} + 2pa-6(m+1) < 0.
\end{equation}
Therefore condition (e) implies that $\gamma_s$ lies above the line $z  - m_s (x+1) = z_-$ near the node.  

Next, we want to show that it remains above the line moving away from $L^-$ in reverse time.  To do so we consider the vector field on lines of the form
\begin{equation*}
	C = z - m_s x.
\end{equation*}
In particular, we look for conditions such that 
\begin{equation}
	\label{cdown}
	\dot{C}|_{C= z_-} \leq 0.
\end{equation}
When this happens, $\gamma_s$ must be repelled away above the line in reverse time.  Obviously, $\dot{C} = 0$ at the node $(-1,z_-)$.  When 
$$p(a+1)>2,$$ then $ \dot{C}|_{C= z_-} $ is increasing as a function of $x$ and the condition in \eqref{cdown} is satisfied.  Thus conditions (e) and (f) together ensure that $\gamma_s$ lies above the line $z  - m_s (x+1) = z_-$ on $M^-_A.$
\end{proof}

We define $z_*$ to be the intersection of the line $x=-2$ (i.e. $P(L^+)$) with the linear approximation of the funnel, $z  - m_s (x+1) = z_-$ as shown in Figure \ref{mmoFig:zStar}.  Lemma \ref{thm:funnel} ensures that $z_*$ lies in the interior of the funnel.  As we construct the singular periodic orbit, $z_*$ provides a target for trajectories returning from $M_A^+$.
 
 \subsection{Singular Periodic Orbit} 
We now seek conditions so that a singular orbit leaves the folded node, lands on $M^+_A$ along $P(L^-)$, follows a trajectory of the reduced problem towards $L^+$, crosses $L^+$ transversely, and returns to $M^-_A$ on $P(L^+)$ below $z_*$.  Singularities on $L^+$ and $M_A^+$ will play a major role in determining conditions that guarantee the existence of the singular periodic orbit.

We will define $z_+$ to be the $z$ coordinate of the folded singularity on $L^+$, so
\begin{align*}
	z_+ &= (m+1) h(1) + \lambda - f(1) \\
	&= (m+1) (k-2) + \lambda - p(1-a)^2 + b.
\end{align*}
If $z_+$ lies above the $z$ nullcline, then there will be a region where trajectories cross $L^+$ transversely as depicted in Figure \ref{mmoFig:m0pp}.  

\begin{figure}[t]
	\begin{center}
	\includegraphics[width=\textwidth]{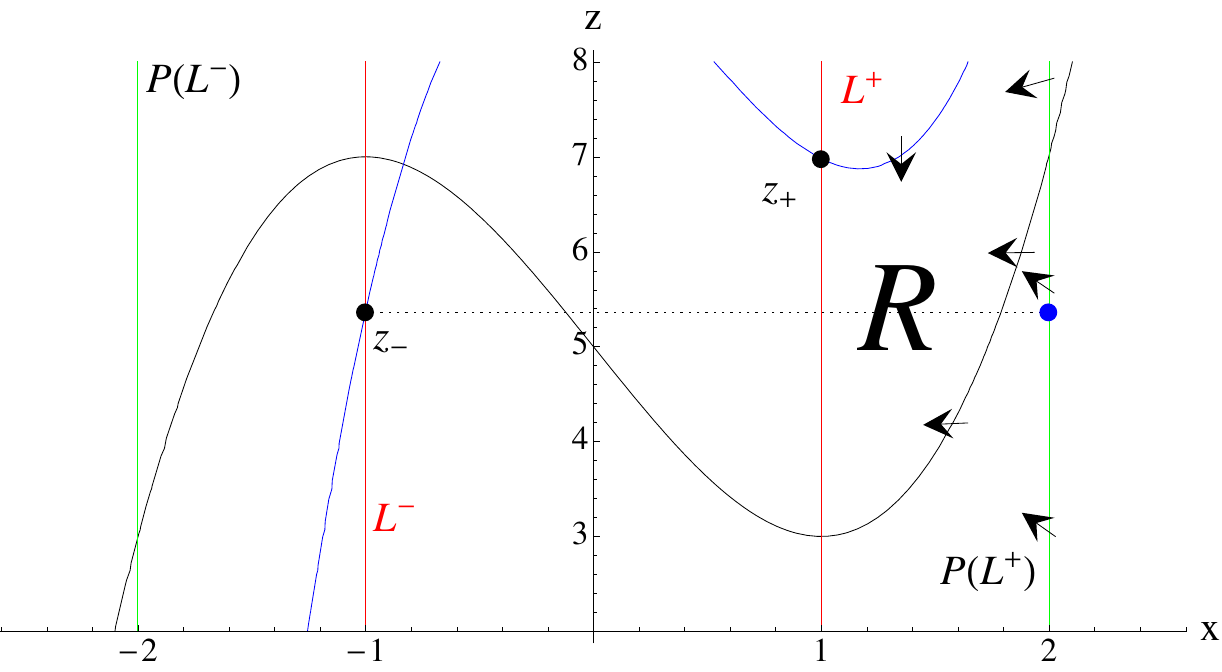}
	\caption{Position of nullclines in the singular limit when $a=0.8, \ p=3, \ b = 2.1, \ k = 4, \ r = 1, \ m= 1,$ and $\lambda=1$.  The blue curve denotes the $x$ nullcline, and the black curve denotes the $z$ nullcline.  The blue point is the landing point of the singular orbit from the node.  The region $R$ between the nullclines on $M_A^+$ is locally positively invariant.}
	\label{mmoFig:m0pp}
	\end{center}
\end{figure}

\begin{lemma}
\label{thm:transverse}
Let equations \eqref{xp}-\eqref{zp} satisfy the conditions of Lemmas \ref{thm:node} and \ref{thm:funnel}.  Furthermore, assume $\delta < 4$ and 
	\begin{enumerate}[(a)]
		\setcounter{enumi}{6}
		\item $4 (a p - m) - \delta > 0$
		\item $\Delta_{\delta} (a, p, m) < 0,$
	\end{enumerate}
	where 
	\begin{align}
		\label{cubDisc}
		\Delta_{\delta} (a, p, m) = & p^2 (-3 m + 2 a p)^2 - 4 m (-3 m + 2 a p)^3 \\
		& + 4 p^3 (-\delta - 2 m + p + 2 a p) - 18 m p (-3 m + 2 a p) (-\delta - 2 m + p + 2 a p) \\
		& - 27 m^2 (-\delta - 2 m + p + 2 a p)^2. \notag
	\end{align}
	Then the singular orbit from the folded node will land on $P(L^-) \subset M^-_A$, follow a trajectory of the reduced problem \eqref{desing}, and cross the fold $L_+$.
\end{lemma}

\begin{remark}
The condition that $\delta < 4$ will be replaced with a stricter condition in Lemma \ref{thm:return} to ensure that the singular orbit returns to the funnel.
\end{remark}

\begin{proof}
The intersection of the $z$ nullcline with $L^+$ occurs at $z = h(1) + \lambda$.  Therefore, the region $R$ between the nullclines on $M^+_A$ will be locally positively invariant if
$$ z_+ = (m+1) h(1) + \lambda -f(1) > h(1) + \lambda, $$ 
which happens if and only if
\begin{align*}
	 0&< m h(1) - f(1) \\
	  \Leftrightarrow 0&< m (k-2) - p (1-a)^2 + b \\
	 \Leftrightarrow  0&<4 (ap - m) - \delta.
\end{align*}
Thus condition (g) gives us that the nullclines are aligned as in Figure \ref{mmoFig:m0pp} along $L^+$, and the positively invariant region exists. Any trajectory that enters $R$ can only escape by crossing $L^+$.  Next, we show that the singular orbit from the node enters $R$.  

The assumption that $\delta < 4$ ensures that $z_- >  h(1) + \lambda$. This is because the fast fiber from the folded node on $L^-$ lands on $P(L^-) \subset M^+_A$ exactly the distance $\delta$ below the $z$ nullcline.  At the landing point (denoted by a blue dot in Figure \ref{mmoFig:m0pp}, the vector field of \eqref{desing} points up and to the left.  If the $x$ nullcline lies above the $z$ nullcline, then the trajectory will continue up and to the left until it enters $R$.  Condition (g) implies that the $x$ nullcline lies above the $z$ nullcline at the fold.  Thus, the only way for the nullclines to switch their orientation is for them to intersect, creating a true equilibrium of \eqref{desing}.  The nullclines intersect wherever the curves $z = h(x) + \lambda$ and $z = (m+1) h(x) + \lambda - f(x)$ intersect.  That is, intersections occur whenever
$$m h(x) - f(x) =0.$$
Note that $m h(x) - f(x)$ is a cubic.  Therefore, the number of zeroes of $m h(x) - f(x)$ is determined by the cubic discriminant, which is precisely the quantity $\Delta_{\delta}.$

If $\Delta_{\delta} < 0$ there is only one intersection, but if $\Delta_{\delta} > 0$ there are three. Condition (c)  implies the $x$ nullcline lies below the $z$ nullcline on $L^-$ (i.e. where $x= -1$), and condition (g) implies the $x$ nullcline lies above the $z$ nullcline on $L^+$ (i.e. where $x=+1$).  By the Intermediate Value Theorem, the nullclines will intersect for some $x$ such that $-1 < x <1$.  Therefore, the conditions (g) and (h) prevent there from being an intersection on either stable branch of $M_0$.  This implies a singular trajectory through the folded node will cross  $L^+$.
\end{proof}

\begin{remark}
In fact, the condition $\delta > 0$ precludes true equilibria on $M_A^-$.  This can be seen by comparing the slopes of the $x$ and $z$ nullclines on $M_A^-$.  The $x$ nullcline is the curve $z = (m+1) h(x) - \lambda - f(x)$, so it has slope
\begin{align*}
	\frac{dz}{dx} &= (m+1) h'(x) - f'(x) \\
	&=(m+1) h'(x) - 2p(x-a)\\
	& > (m+1) h'(x),
\end{align*}
since $x \leq-1$ on $M_A^-$.  Meanwhile, the $z$ nullcline is given by the equation $z = h(x) + \lambda$ which has slope
$$ \frac{dz}{dx} = h'(x).$$
Since an equilibrium is precisely the intersection of these curves, any equilibrium on $M_A^-$ will result in the $x$ nullcline crossing the fold above the $z$ nullcline, implying $\delta < 0$.  
\end{remark}

Lemma \ref{thm:transverse} allows for the possibility that the folded singularity $(1,z_+)$ is also a folded node.  If we consider the Jacobian at the point $(1,z_+)$, we see that condition (g) implies $\det(J(1,z_+))>0$.  Therefore, the stability of the folded singularity depends on $f'(1)$.  To exclude the possibility of SAOs along $L^+$, we want to avoid the case where $(1,z_+)$ is a stable folded node.  If $f'(1) > 0,$ then the folded singularity will be unstable.  Requiring $f'(-1) < 0 < f'(1)$ implies that $p>0$ and $-1 < a < 1$.  We update condition (b) from Lemma \ref{thm:node} accordingly, so we now have
$$ (b) \ \ -1< a < 1.$$

Finally, we need to find conditions so that the singular trajectory from the folded node returns to the funnel.  This will show that we in fact have a singular periodic orbit.

\begin{lemma}
	\label{thm:return}
	Let equations \eqref{xp}-\eqref{zp} satisfy the conditions (a)-(h) from Lemmas \ref{thm:node}-\ref{thm:transverse}.  Additionally, suppose the equations satisfy 
	\begin{enumerate}[(a)]
		\setcounter{enumi}{8}
		\item $4 (m+4) - 5ap - p > 0.$
	\end{enumerate}
	Then there is a singular periodic orbit $\Gamma.$
\end{lemma}

\begin{proof}
By Lemma \ref{thm:node}, we know the system will have a folded node singularity.  Let $\Gamma$ be the singular trajectory consisting of the fast fiber of the layer problem from the singular node to $P(L^-)$.  By Lemma \ref{thm:transverse} we know that the trajectory will follow the slow flow on $M^+_A$ until it crosses $L_+$.  Furthermore, we know that $z_+$ is an upper bound on the $z$ coordinate of the intersection.  If $z_+ < z_*$,  then $\Gamma$ will land in the singular funnel upon leaving $L^+$.  Direct calculation shows that $z_+ < z_*$ precisely when $4 (m+4) - 5ap - p > 0.$
\end{proof}

\subsection{Main Result}
\begin{theorem}
	\label{thm:mainresult}
	Suppose the parameters of the system \eqref{xp}-\eqref{zp} satisfy the conditions (a)-(i).	Then, for $\epsilon$ sufficiently small, the system will have a stable periodic orbit of MMO-type $1^s$ for some $s>0$.  
\end{theorem}

\begin{proof}
Lemmas \ref{thm:node}-\ref{thm:return} show that these conditions satisfy the assumptions of Theorem \ref{thm:assumptions}.
\end{proof}

\begin{figure}
        \centering
        \begin{subfigure}[t]{0.45\textwidth}
               \includegraphics[width=\textwidth]{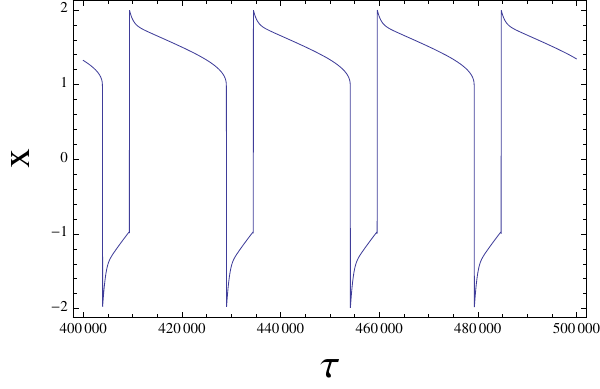}
		\caption{Model output for $x$.}
		\label{mmoFig:output3}
        \end{subfigure}
        ~ 
        \begin{subfigure}[t]{0.45\textwidth}
                \includegraphics[width=\textwidth]{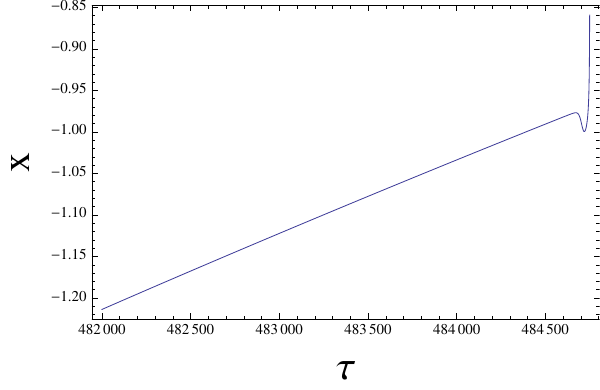}
		\caption{A closer look at the small amplitude oscillations in Figure \ref{mmoFig:output3}}
		\label{mmoFig:sao3}
        \end{subfigure}
        \caption{MMO orbit for $\epsilon=0.001$, $a=0.91$, $p=1.05$, $b=0.31$, $k=2.2$, $r=0.3$, $\lambda=1$, and $m=0.6$.  With these parameters $\delta=1$.}
        \label{mmoFig:e3}
\end{figure}

Figure \ref{mmoFig:apmConc} depicts a portion of parameter space that satisfies conditions (a)-(i) in Theorem \ref{thm:mainresult}.  These conditions place restrictions on $a, \ p, \ m,$ and $r$ explicitly, as well as $b$ and $k$ through the restrictions on $\delta.$  However, there are no restrictions on $\lambda$.  Additionally, Figure \ref{mmoFig:e3} shows the time series for $x$ for a trajectory satisfying the conditions of Theorem \ref{thm:mainresult}.  

\begin{figure}
        \centering
        \begin{subfigure}[b]{\textwidth}
        \centering
                \includegraphics[width=0.6\textwidth]{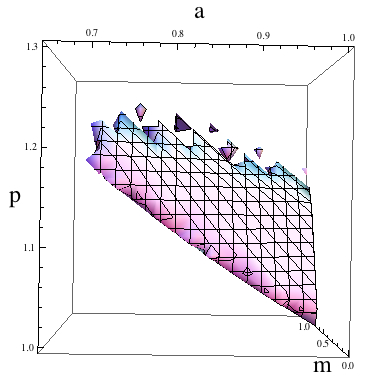}
                \caption{Solid in $apm$-space.}
                \label{mmoFig:apm1}
        \end{subfigure}%
        ~ 
          \\
        \begin{subfigure}[b]{0.45\textwidth}
                \includegraphics[width=\textwidth]{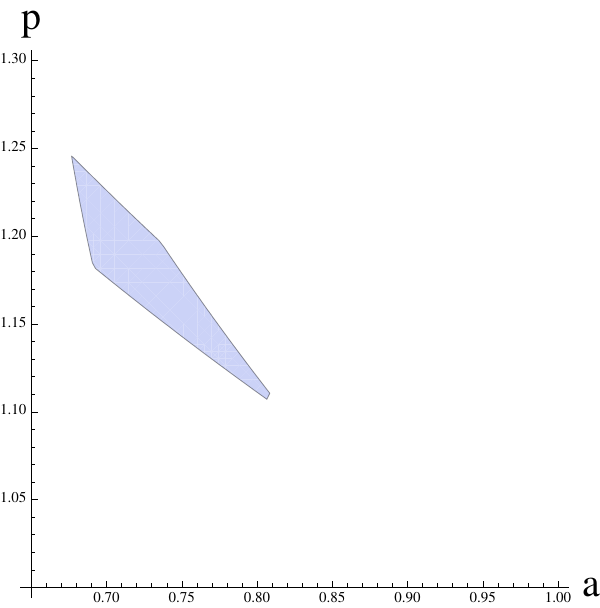}
                \caption{Slice for $m=0.4$.}
                \label{mmoFig:ap14}
        \end{subfigure}
        ~ 
        \begin{subfigure}[b]{0.45\textwidth}
                \includegraphics[width=\textwidth]{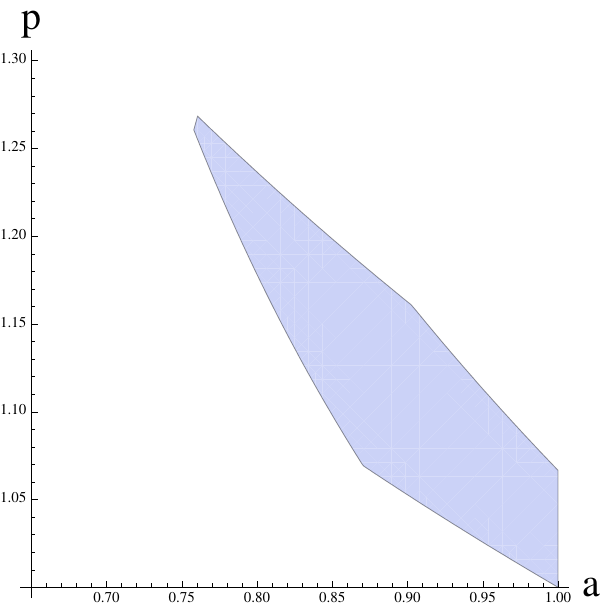}
                \caption{Slice for $m=0.6$}
                \label{mmoFig:ap16}
        \end{subfigure}
        \caption{Parameters n $apm$-space for $\delta=1.3$ that satisfy conditions (a)-(i) from Theorem \ref{thm:mainresult}. }
        \label{mmoFig:apmConc}
\end{figure}

\subsection{Extending the Parameter Regime}
While the conditions (a)-(i) in Theorem \ref{thm:mainresult} are sufficient, they are not all necessary conditions for the model to exhibit MMOs.  In fact, they are rather strict.  This is a direct consequence of linearly approximating the funnel to obtain conditions analytically.  Figure \ref{mmoFig:apmNC} depicts the portion of phase space satisfying only the conditions of Theorem \ref{thm:mainresult} that do not relate to the linear approximation of the funnel.  However, not all parameters from the region pictured in Figure \ref{mmoFig:apmNC} will produce MMO orbits.  

The stable periodic orbits (of some MMO type) outside of the parameter regime described by Theorem \ref{thm:mainresult}, such as the one in Figure \ref{mmoFig:timeseries}, can be much more complicated as a result of the return mechanism projecting the singular periodic orbit closer to the boundary of the funnel (i.e., closer to the strong canard $\gamma_s$).  The behavior in this regime is also described by  Br{\o}ns {\it et al} in \cite{brons}.  For the parameters that generate the orbit in Figure \ref{mmoFig:timeseries}, $\mu \approx 0.1010$ and $s=5$.  However, the MMO signature for the orbit is $1^2$.  This is because the return mechanism sends the trajectory near the boundary of the funnel in the singular limit.  

\begin{figure}
        \centering
        \begin{subfigure}[b]{\textwidth}
        \centering
                \includegraphics[width=0.6\textwidth]{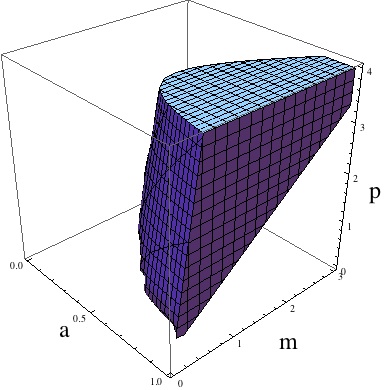}
                \caption{Solid in $apm$-space.}
                \label{mmoFig:apm2}
        \end{subfigure}%
        ~ 
          \\
        \begin{subfigure}[b]{0.3\textwidth}
                \includegraphics[width=\textwidth]{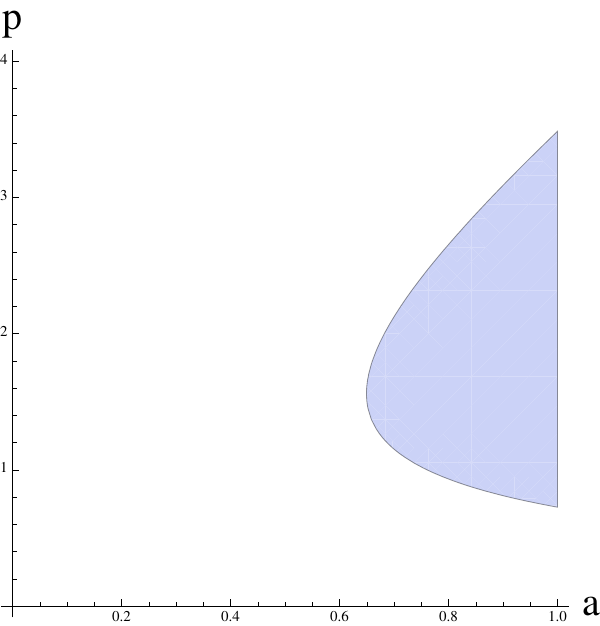}
                \caption{Slice for $m=0.4$.}
                \label{mmoFig:apNC04}
        \end{subfigure}
        ~ 
        \begin{subfigure}[b]{0.3\textwidth}
                \includegraphics[width=\textwidth]{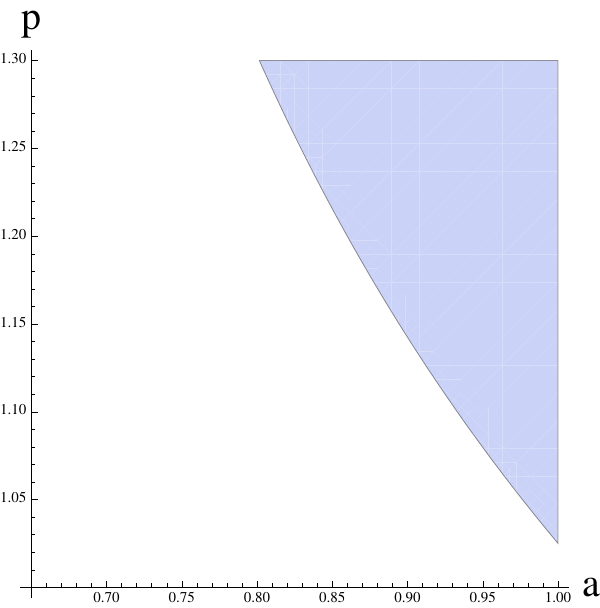}
                \caption{Slice for $m=0.7$}
                \label{mmoFig:apNC07}
        \end{subfigure}
           ~ 
        \begin{subfigure}[b]{0.3\textwidth}
                \includegraphics[width=\textwidth]{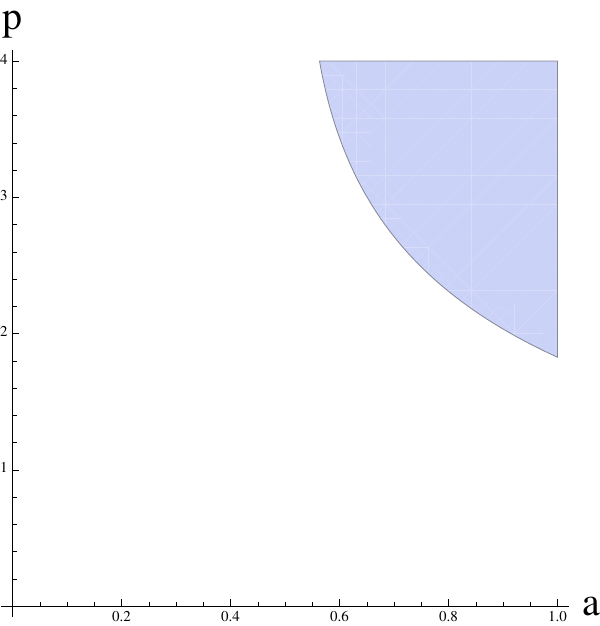}
                \caption{Slice for $m=1.5$}
                \label{mmoFig:apNC15}
        \end{subfigure}
        \caption{Parameters in $apm$-space for $\delta=1.3$ that satisfy conditions (a)-(d), (g), and (h) from Theorem \ref{thm:mainresult}. }
        \label{mmoFig:apmNC}
\end{figure}

\section{Discussion}
We have found sufficient conditions such that the system \eqref{xp}-\eqref{zp} has a stable periodic orbit with MMO signature $1^s$.  To our knowledge, this is the first climate-based model that has been analyzed to demonstrate MMOs.  The dimensionless model is a variant of the Koper model with an added nonlinearity.  As with the standard Koper model, the model has an `S'-shaped critical manifold and a parameter regime with both a folded node and global return mechanism.  Although the additional nonlinearity in the model does not factor into obtaining a folded node, nonlinear effects play a significant role in determining the shape of the funnel, and consequently the return mechanism.  From a mathematical standpoint, it is significant that the additional nonlinearity does not destroy the functionality of the model to produce an MMO pattern.  

\begin{figure}
        \centering
        \begin{subfigure}[b]{0.7\textwidth}
        \centering
                \includegraphics[width=0.6\textwidth]{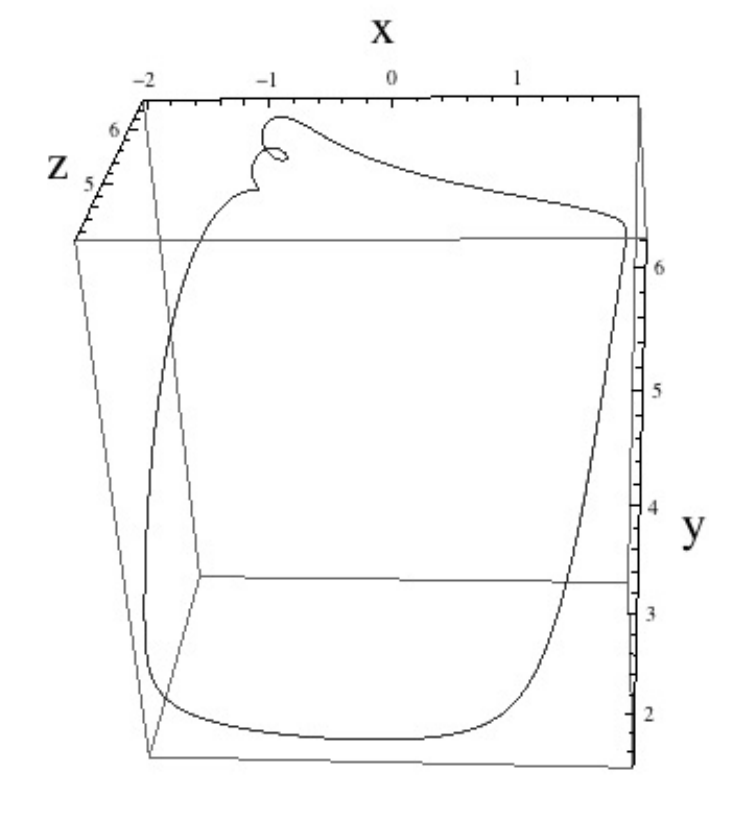}
		\caption{Attracting periodic orbit in the 3D phase space. }
		\label{mmoFig:mmo3d}
        \end{subfigure}%
        \\ 
        \begin{subfigure}[b]{0.45\textwidth}
                \includegraphics[width=\textwidth]{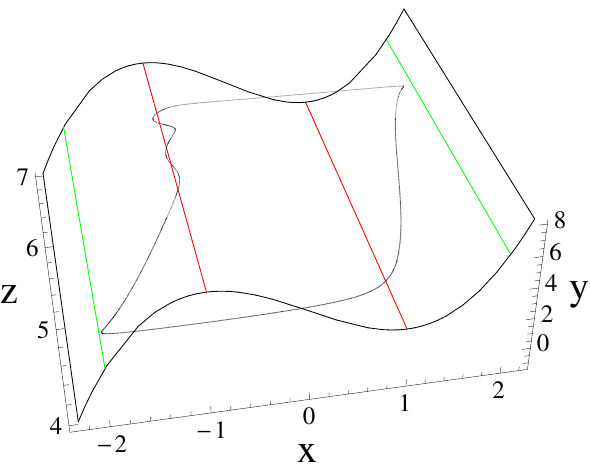}
		\caption{The attracting periodic orbit shown with the critical manifold.}
		\label{mmoFig:trajectory}
        \end{subfigure}
        ~ 
        \begin{subfigure}[b]{0.45\textwidth}
                \includegraphics[width=\textwidth]{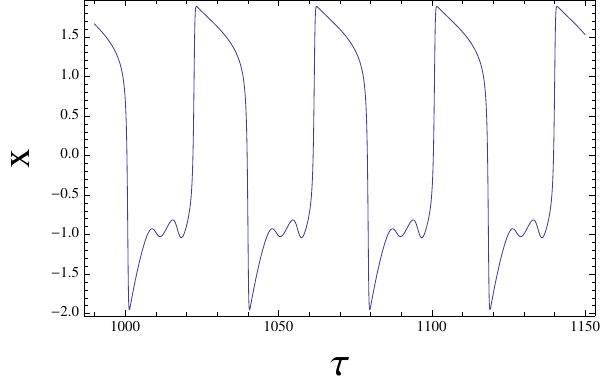}
		\caption{Model output for $x$ for the trajectory in Figure \ref{mmoFig:mmo3d}.}
		\label{mmoFig:output}
        \end{subfigure}
        \caption{MMOs for $\epsilon = 0.1,$ $a = 0.8$, $p = 3,$ $b= 2.1$, $k=4,$ $r=1$, $m=1$, and $\lambda = 1$. }
        \label{mmoFig:e1}
\end{figure}

We are able to find the conditions in Theorem \ref{thm:mainresult} analytically, which is a rarity for MMO problems.  Although it is nice to have an analytical proof, the approach excludes a significant region of parameter space where MMOs can be found.  Numerically approximating the strong canard, the standard practice for demonstrating MMOs, helps provide a more complete picture of the parameter regime that produces MMOs.  The method relies on varying one parameter at a time, making it difficult to actually plot the complete region.

The time series in Figures \ref{mmoFig:tempRecord} and \ref{mmoFig:output} are qualitatively similar in that they both contain large oscillations followed by a series of smaller amplitude oscillations.  Note that $\epsilon=0.1$ in Figure \ref{mmoFig:e1} is not truly small.  In Figures \ref{mmoFig:examples1} and \ref{mmoFig:timeseries} we plotted analogous trajectories with smaller $\epsilon$ (keeping the other parameters fixed).  Figure \ref{mmoFig:singularOrbit} depicts the singular orbit for all of these examples.  We focus on the ``nice'' MMO in Figure \ref{mmoFig:e1} because of its similarity to Figure \ref{mmoFig:tempRecord}.  Since $\epsilon$ is larger, we are able to see the small amplitude oscillations clearly.  

The model can give us some insight about the climate system.  The physical implication of the requirement that $-1 < a < 1$ is that $\text{CO}_2$ drawdown due to terrestrial mechanisms is most efficient at a temperature (or ice volume) somewhere between the stable glacial and interglacial states.  Through the requirements on $\delta$ we learn about the relationship between $b,$ $m,$ and $k$.  This relates the amount of $\text{CO}_2$ removed from the atmosphere when the planet is most efficient at doing so ($b$), the ratio of the timescales of the land-atmosphere carbon flux to that of the ocean-atmosphere exchange ($m$), and the minimum/maximum values of atmospheric carbon ($k$).  Finally, $r$ tells us something about the proportion of carbon in the atmosphere to carbon in the ocean required for the ocean to switch from absorbing to outgassing.  If $r$ is large, we will no longer have a folded node.  It may be the case that $r \ll 1$, which puts us near the folded saddle-node limit and allows for more complicated behavior.  Some simulations with $r \ll 1$ are shown in Figure \ref{mmoFig:threeScales}.  

\begin{figure}
        \centering
        \begin{subfigure}[b]{0.3\textwidth}
        \centering
                \includegraphics[width=\textwidth]{threeScales1.pdf}
		\caption{Example of 3 time-scale series $\epsilon = 0.1,$ $a = 0.8$, $p = 3,$ $b= 2$, $k=4,$ $r=0.05$, $m=1$, and $\lambda = 1$.}
        \end{subfigure}%
        ~ 
        \begin{subfigure}[b]{0.3\textwidth}
                \includegraphics[width=\textwidth]{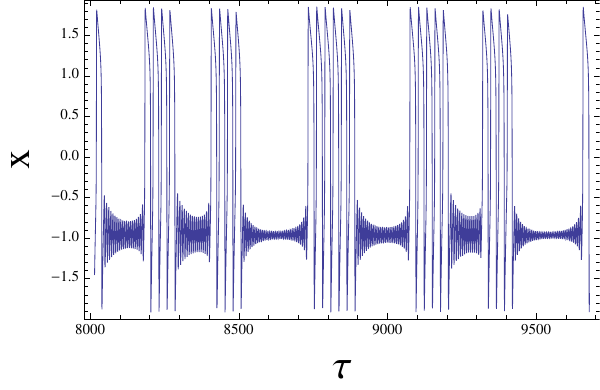}
		\caption{Example of 3 time-scale series $\epsilon = 0.1,$ $a = 0.8$, $p = 3,$ $b= 2.3$, $k=4,$ $r=0.01$, $m=1$, and $\lambda = 1$.}
        \end{subfigure}
        ~ 
        \begin{subfigure}[b]{0.3\textwidth}
               \includegraphics[width=\textwidth]{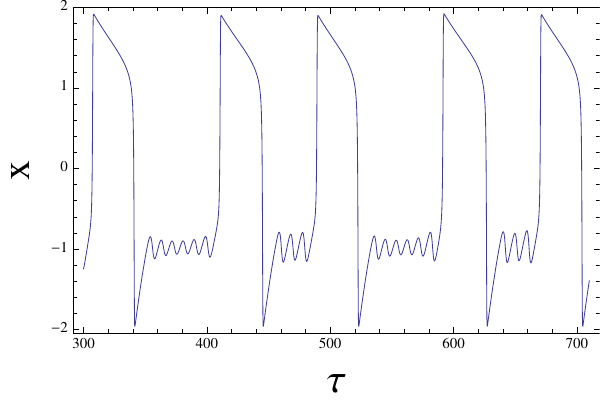}
		\caption{Example of 3 time-scale series $\epsilon = 0.05,$ $a = 0.8$, $p = 3,$ $b= 2.32$, $k=4,$ $r=0.1$, $m=1$, and $\lambda = 1$.}
		\label{mmoFig:c3scales}
        \end{subfigure}
        \caption{Examples of MMO patterns in the three time-scale case. }
        \label{mmoFig:threeScales}
\end{figure}

The analysis required to show MMOs due to a folded node assumes a separation of time scales and (at least) two slow variables.  As mentioned in the introduction and Section 2, it is often difficult to determine exactly which parameters are small enough to perform this analysis.  Here we rely on the wisdom of climate scientists.  It may be that changes in atmospheric greenhouse gases happen on a similar timescale to temperature.  Figure \ref{mmoFig:e1} depicts the case where there is only a marginal time-scale separation and we still see MMOs.  

The power of canard theory lies in its generality. It can be applied to a diverse range of research areas from mathematical physiology, fluid dynamics, magnetohydrodynamics and even climate modelling. It was recently applied to explain the `compost bomb instability' - a potentially catastrophic explosive release of peatland soil carbon into the atmosphere as the greenhouse gas carbon dioxide, which could significantly accelerate anthropogenic global warming \cite{wieczorek11}. The take-home message lies in the realization that folded singularities and associated canards create local transient `attractor' states in multiple scales problems. This is due to the fact that trajectories in the domain of attraction of folded singularities will reach and pass these folded singularities in finite slow time; folded singularities are not equilibrium states. In the context of climate tipping point problems, \cite{wieczorek11} identifies {\it canards of folded saddle type} as {\it threshold manifolds}, while \cite{mitry13} and \cite{wechs13} identify the same canard structure as firing threshold manifolds in the context of neural excitability.

\section*{Acknowledgements}
A.R. and C.J. were supported by the Mathematics and Climate Research Network on NSF grants DMS-0940363 and DMS-1239013.  E.W. was supported in part by the MCRN on NSF grant DMS-0940363 as well as a grant to the University of Arizona from HHMI (52006942).  M.W. was supported by the Australian Research Counsel.  Additionally, we would like to thank Sebastian Wieczorek for his helpful insight in the beginning stages of this project and Michel Crucifix for his helpful comments during the review process.


\nocite{gsp}
\bibliographystyle{amsplain}
\bibliography{sources}

\end{document}